\documentclass[12pt]{amsart}
\usepackage{amsmath, amssymb, amsthm, latexsym}
\usepackage{amssymb}
\usepackage{latexsym}
\usepackage[center]{caption}
\usepackage{tikz}
\newcounter{braid}
\newcounter{strands}

\pgfkeyssetvalue{/tikz/braid height}{1cm}
\pgfkeyssetvalue{/tikz/braid width}{1cm}
\pgfkeyssetvalue{/tikz/braid start}{(0,0)}
\pgfkeyssetvalue{/tikz/braid colour}{black}
\pgfkeys{/tikz/strands/.code={\setcounter{strands}{#1}}}

\makeatletter
\def\cross{%
  \@ifnextchar^{\message{Got sup}\cross@sup}{\cross@sub}}

\def\cross@sup^#1_#2{\render@cross{#2}{#1}}

\def\cross@sub_#1{\@ifnextchar^{\cross@@sub{#1}}{\render@cross{#1}{1}}}

\def\cross@@sub#1^#2{\render@cross{#1}{#2}}

\def\render@cross#1#2{
  \def\strand{#1}
  \def\crossing{#2}
  \pgfmathsetmacro{\cross@y}{-\value{braid}*\braid@h}
  \pgfmathtruncatemacro{\nextstrand}{#1+1}
  \foreach \thread in {1,...,\value{strands}}
  {
    \pgfmathsetmacro{\strand@x}{\thread * \braid@w}
    \ifnum\thread=\strand
    \pgfmathsetmacro{\over@x}{\strand * \braid@w + .5*(1 - \crossing) * \braid@w}
    \pgfmathsetmacro{\under@x}{\strand * \braid@w + .5*(1 + \crossing) * \braid@w}
    \draw[braid] \pgfkeysvalueof{/tikz/braid start} +(\under@x pt,\cross@y pt) to[out=-90,in=90] +(\over@x pt,\cross@y pt -\braid@h);
    \draw[braid] \pgfkeysvalueof{/tikz/braid start} +(\over@x pt,\cross@y pt) to[out=-90,in=90] +(\under@x pt,\cross@y pt -\braid@h);
    \else
    \ifnum\thread=\nextstrand
    \else
     \draw[braid] \pgfkeysvalueof{/tikz/braid start} ++(\strand@x pt,\cross@y pt) -- ++(0,-\braid@h);
    \fi
   \fi
  }
  \stepcounter{braid}
}

\tikzset{braid/.style={double=\pgfkeysvalueof{/tikz/braid colour},double distance=1pt,line width=2pt,white}}

\newcommand{\braid}[2][]{%
  \begingroup
  \pgfkeys{/tikz/strands=2}
  \tikzset{#1}
  \pgfkeysgetvalue{/tikz/braid width}{\braid@w}
  \pgfkeysgetvalue{/tikz/braid height}{\braid@h}
  \setcounter{braid}{0}
  \let\sigma=\cross
  #2
  \endgroup
}
\makeatother

\input xypic
\newtheorem{theorem}{Theorem}%[section]
\newtheorem{proposition}[theorem]{Proposition}

\newtheorem{lemma}[theorem]{Lemma}

\def\Z{\mathbb{Z}}

\def\C{\mathbb{C}}

\def\Q{\mathbb{Q}}

\def\R{\mathbb{R}}
\def\C{\mathbb{C}}

\def\Pi{\mathbb{P}^{\infty}}

\def\qed{\hfill$\square$\medskip}

\def\Zpk{\mathbb{Z}/p^{k}}
\def\Zpk1{\mathbb{Z}/p^{k-1}}

\newcommand{\rref}[1]{(\ref{#1})}

\newcommand{\beg}[2]{\begin{equation}\label{#1}#2\end{equation}}
\def\r{\rightarrow}

\def\sl2{\widetilde{SL_{2}(\Z)}}

\title[Equivariant K-theory of groups with involution]{Equivariant K-theory of compact Lie groups with involution}
\author{Po Hu, Igor Kriz and Petr Somberg}
\thanks{The authors acknowledge the support of the Eduard \v{C}ech ECE in Prague GA CR P201/12/G028,
and NSF grants DMS 1104348 and DMS 1102614}
\begin{document}

\maketitle

\begin{abstract}
For a compact simply connected simple Lie group $G$ with an involution $\alpha$, we compute 
the $G\rtimes \Z/2$-equivariant K-theory of $G$ where $G$ acts by conjugation and $\Z/2$
acts either by $\alpha$ or by $g\mapsto \alpha(g)^{-1}$. We also give a representation-theoretic
interpretation of those groups, as well as of $K_G(G)$.
\end{abstract}

\section{Introduction}\label{intro}

Brylinski and Zhang \cite{brylinski} computed, for (say) a simple compact Lie group $G$,
the $G$-equivariant K-theory ring $K_G(G)$ with $G$ acting on itself by conjugation,
as the ring $\Omega_{R(G)/\Z}$ of K\"{a}hler differentials of $R(G)$ over $\Z$ 
(see also Adem, Gomez \cite{ag} for related work).
Let $\alpha$ be an involutive automorphism of $G$. Then we can consider
actions of $G\rtimes \Z/2$ on $G$ where $G$ acts by conjugation and the
generator of $\Z/2$ acts
either by the automorphism $\alpha$ or by the map $\gamma:g\mapsto \alpha(g)^{-1}$.
The main result of the present paper is a computation of $K_{G\rtimes \Z/2}(G)$
in both cases as a module over $K_{G\rtimes \Z/2}(*)=R(G\rtimes\Z/2)$.
(Here $R(H)$ is the complex representation ring of a compact Lie group $H$.)

The involutive automorphism $\alpha$ determines a compact symmetric space $G/G^\alpha$, 
and we were originally interested in these computations as a kind of topological invariant
of symmetric pairs of compact type. (Recall, in effect, that $G/G^\alpha$ is a connected
component of $G^\gamma$ via the embedding $x\mapsto \alpha(x)x^{-1}$.)
It turns out, however, that the groups $K_{G\rtimes \Z/2}(G)$ are a rather crude
invariant of symmetric pairs, since they essentially only depend on whether $\alpha$ is
an outer or inner automorphism of $G$; if $\alpha$ is an inner automorphism,
$G\rtimes \Z/2$ becomes a central product, which behaves essentially the same as
the direct product from our point of view.

Nevertheless, having a complete calculation is still interesting, as are some of the methods
involved in it. The main ingredient of the method we present here is the construction of
Brylinski-Zhang \cite{brylinski,brylinski1} of the element $dv\in\Omega_{R(G)/\Z}=K_G(G)$
for a finite-dimensional complex representation $v$ of $G$. That construction,
unfortunately, was presented incorrectly in \cite{brylinski} (in fact, the elements
written there are $0$), and so we developed an
alternate construction of those elements using induction from the normalizer of a maximal
torus. However, Brylinski \cite{brylinski1} communicated the correct construction to us. 
The construction \cite{brylinski1} is completely geometric, and supersedes our previous
induction method (which, for that reason, we omit from this presentation). 
In fact, the construction \cite{brylinski1} turns out to be equivariant with respect
to both the $\alpha$ and $\gamma$ actions. This allows an ``obvious guess" of what
$K_{G\rtimes\Z/2}(G)$ should be. We validate that guess following the methods 
of Greenlees and May \cite{gmt}, involving Tate cohomology. (Essentially,
the main point is that under suitable finiteness hypothesis, a $\Z/2$-equivariant
map of $\Z/2$-CW complexes
which is an equivalence non-equivariantly is an equivalence equivariantly because
the Tate cohomology can be computed as an algebraic functor of the ``geometric
fixed points''.)

We realized, however, that the construction \cite{brylinski1} can be 
generalized to give a representation-theoretical interpretation of the groups
$K_G(G)$, $K_{G\rtimes \Z/2}(G)$. Such an interpretation is strongly motivated
by the work of Freed, Hopkins and Teleman \cite{fht} who showed that if $\tau$
is a regular $G$-equivariant twisting of $K$-theory on $G$, then the twisted
equivariant K-theory $K_{G,\tau}(G)$ is isomorphic to the free abelian group on
irreducible projective representations 
of level $\tau-h^\vee$ (where $h^\vee$ is the dual Coxeter number) 
of the loop group $LG$. 

This suggests that untwisted K-theory $K_G(G)$ should correspond to representations
at the critical level of the Lie algebra $L\frak{g}$. We found that this is indeed true,
but the representations one encounters are not lowest weight representations
(which occur, for example, in the geometric Langlands program). Instead, the 
fixed point space of the infinite loop space
$K_G(G)_0$ turns out to be the group completion of the space of 
{\em finite} representations of the loop group $LG$ with an appropriate
topology. Here by a finite representation
we mean a finite-dimensional representation which factors through a projection
$LG\r G^n$ given by evaluation at finitely many points (cf. \cite{ps}). (It is possible to conjecture
that every finite-dimensional representation of $LG$ is finite, although it may
depend on the type of loops we consider; in this paper, we restrict our attention
to continuous loops.) In fact, we also prove that this is true $\Z/2$-equivariantly
with respect to involutions, i.e. that the fixed point space of $K_{G\rtimes\Z/2}(G)_0$ is
the group completion of the space of representations of $LG\rtimes \Z/2$,
where $\Z/2$ acts on $LG$ via its action on $G$ in the case of $\Z/2$ acting
on $G$ by $\alpha$, and simultaneously on $G$ and on the loop parameter
by reversal of direction in the case when $\Z/2$ acts on $G$ by $\gamma$.

The present paper is organized as follows: In Section \ref{ssi}, we review the construction
of Brylinski-Zhang \cite{brylinski,brylinski1} and study its properties with respect to
the involution on $G$. In Section \ref{smi},
we compute the $R(G\rtimes \Z/2)$-modules $K^{*}_{G\rtimes \Z/2}(G)$. 
In Section \ref{sconc}, we discuss the computation in more concrete terms,
and give some examples. In Section \ref{srep}, we give an interpretation of
$K^{*}_{G}(G)$ in terms of representation of the loop group $LG$, and
in Section \ref{srep2}, we make that interpretation $\Z/2$-equivariant,
thus extending it to $K^{*}_{G\rtimes \Z/2}(G)$.

\vspace{3mm}

\section{The Brylinski-Zhang construction}
\label{ssi}

Let $G$ be a simply connected compact Lie group, $T$ a maximal torus, $N$ its
normalizer, $W$ the Weyl group. Let $R(G)$ denote, as usual, the
complex representation ring. Recall that if $u_1,\dots, u_n$ are the fundamental
weights of $G$ ($n=rank(G)$), then the weight lattice $T^*=Hom(T,S^1)$
is freely generated by $u_1,\dots, u_n$ and we have
$$R(T)=\Z[u_1,u_{1}^{-1},\dots, u_n,u_{n}^{-1}],$$
$$R(T)\supset R(G)=\Z[\overline{u_1},\dots, \overline{u_n}]$$
where $\overline{u_i}$ is the sum of elements of the $W$-orbit of $u_i$.

Let, for a map of commutative rings $S\r R$, $\Omega_{R/S}$ denote
the ring of K\"{a}hler differentials of $R$ over $S$.
Then one easily sees by the K\"{u}nneth theorem that we have an 
isomorphism:
\beg{ebz1}{\begin{array}{l}K^{*}_{T}(T)=\Omega_{R(T)/\Z}\\
=\Z[u_1,u_{1}^{-1},\dots,u_n,u_{n}^{-1}]\otimes \Lambda_\Z[du_1,\dots,du_n]
\end{array}
}

\begin{theorem}\label{tbz}
(Brylinski-Zhang \cite{brylinski})
Suppose $G$ acts on itself by conjugation. Then
there is a commutative diagram of rings
\beg{ebz2}{
\diagram
K_G(G)\dto_\cong \rrto^{res.} && K_T(T)\dto^\cong\\
\Omega_{R(G)/\Z}\dto_=\rrto && \Omega_{R(T)/\Z}\dto^=\\
\Z[\overline{u_1},\dots,\overline{u_n}]
\otimes \Lambda [d\overline{u_1},\dots,d\overline{u_n}]
\rrto^{\subset}&&\Z[u_{1}^{\pm 1},\dots,u_{n}^{\pm 1}]\otimes \Lambda[du_1,\dots,du_n].
\enddiagram
}
Moreover, the isomorphisms in \rref{ebz2} can be chosen in such a way
that the generator $dW\in K_G(G)$ for a complex (finite-dimensional) 
$G$-representation $W$ is represented in $G$-equivariant $K$-theory
by a complex of $G$-bundles
\beg{ebz2a}{\diagram
G\times \R\times W\rto^\phi & G\times\R\times W
\enddiagram
}
where $\phi$ is a homomorphism which is iso outside of $G\times \{0\}\times W$, 
and is given by
\beg{ebz3}{\begin{array}{ll}
\phi(g,t,w)=(g,t,tw) & \text{for $t<0$}\\
\phi(g,t,w)=(g,t,-tg(w)) & \text{for $t\geq 0$}.
\end{array}
}
\end{theorem}

\begin{proof}
The proof of the diagram \rref{ebz2} is given in \cite{brylinski}. Since there is a mistake 
in the formula \rref{ebz3} in \cite{brylinski}, (corrected in \cite{brylinski1}), we
 give a proof here. In view of the commutativity of diagram \rref{ebz2},
and injectivity of the diagonal arrows, it suffices to prove the statement for $T$
instead of $G$. By K\"{u}nneth's theorem, it suffices to further consider $T=S^1$.

In the case $T=S^1$, let $z$ be the tautological $1$-dimensional complex representation
of $S^1$ (considered as the unit circle in $\C$). Then the element of
$\widetilde{K}_{S^1}^{-1}(S^1)$ given by \rref{ebz3}
for $W=z^n$ is equal to the element of $\widetilde{K}^{0}_{S^1}(S^2)$ 
given by $H^n-1$ where $H$ is the tautological bundle on $S^2=\C P^1$ 
(with trivial action of $S^1$). But it is well known that
$$H=u+1$$
where $u\in \widetilde{K}^{0}_{S^1}(S^2)$ is the Bott periodicity element, and thus,
(recalling that $u^2=0$), 
$$H^n-1=(u+1)^n-1=nu.$$
Thus, choosing the Bott element as $u$ gives the required isomorphism in the right hand
column of \rref{ebz2} for $T=S^1$, and thus the statement follows.
\end{proof}

\vspace{3mm}
\begin{proposition}
\label{pl2}
Let $G$ be as above, let $\alpha$ be an involutive automorphism of $G$,
and let $W$ be a finite-dimensional complex representation such that 
\beg{ebz6}{\alpha^*(W)\cong W}
(where $\alpha^*(W)$ is the representation of $G$ on $W$ 
composed with the automorphism $\alpha$).
Then, given the choices described in Theorem \ref{tbz}, if the generator
$a$ of $\Z/2$ acts on $G$ by $\alpha$, then $dW$ is in the image of the
restriction (forgetful map)
\beg{ebz4}{K^{1}_{G\rtimes\Z/2}(G)\r K^{1}_{G}(G).
}
When the generator $a$ of $\Z/2$ acts on $G$ by $\gamma$, $dW$ is in the
image of the restriction (forgetful map)
\beg{ebz5}{K^{A}_{G\rtimes \Z/2}(G)\r K^{1}_{G}(G)
}
where $A$ is the $1$-dimensional real representation of $G\rtimes \Z/2$ given by
the sign representation of the quotient $\Z/2$.
\end{proposition}

\begin{proof}
Recall that when \rref{ebz6} holds, then a choice of the isomorphism \rref{ebz6}
can be made to give a representation of $G\rtimes\Z/2$ on $W$. Moreover, there are
precisely two such choices, differing by tensoring by the complex $1$-dimensional
sign representation of $\Z/2$.

Consider first the case when the generator $a$ of $\Z/2$ acts on $G$ by $\alpha$.
Then consider the $\Z/2$-action on the Brylinski-Zhang construction 
\beg{eddiag}{\diagram
G\times \R\times W\dto \rto^\phi &
G\times \R\times W\dto \\
G\times\R\times W\rto^\phi &
G\times\R\times W 
\enddiagram
}
where the generator of $\Z/2$ acts by
$$
\diagram
(g,t,w)\dto|<\stop &(g,t,w)\dto|<\stop \\
(\alpha(g),t,\alpha(w))&
(\alpha(g),t,\alpha(w)).
\enddiagram
$$
When the generator $a$ of $\Z/2$ acts on $G$ by $\gamma$,
consider the $\Z/2$-action \rref{eddiag} where the generator of $\Z/2$ acts  by
$$
\diagram
(g,t,w)\dto|<\stop  &(g,t,w)\dto|<\stop \\
(\alpha(g)^{-1},-t,\alpha(w)) &
(\alpha(g)^{-1},-t,\alpha(g^{-1}w))
\enddiagram
$$
(Note that $\alpha(\alpha(g)\alpha(g^{-1}w))=w$, so the action of $a$ on the
right hand side is involutive. One readily sees that it also intertwines the
action of $G$ via the automorphism $\alpha$.)

To verify that the homomorphism $\phi$ commutes with the involution in 
the case of $a$ acting on $G$ via $\gamma$, since we already know the
action is involutive, it suffices to consider $t<0$. In this case, we have
$$
\diagram
(g,t,w)\rto|<\stop
\dto|<\stop &
(g,t,tw)\dto|<\stop\\
(\alpha(g)^{-1},-t,\alpha(w))\rto|<\stop &
(\alpha(g)^{-1},-t,t\alpha(g)^{-1}\alpha(w)).
\enddiagram
$$
\end{proof}

\vspace{3mm}

\section{The computation of equivariant $K$-theory}
\label{smi}

In this section, we will compute $K_{G\rtimes \Z/2}(G)$ where the generator
$a$ of $\Z/2$ acts by $\alpha$ or $\gamma$. First observe that in both cases, the
generator $a$ of $\Z/2$ acts on $K^{*}_{G}(G)\cong \Omega_{R(G)/\Z}$ by automorphisms
of rings. The action on $R(G)$ is given by a permutation representation given by the permutation
of irreducible representations by the automorphism $\alpha$. Alternately, one may think in terms
of the action of $\alpha$ on Weyl group orbits of weights. Let $u_1,\dots,u_n$ be the fundamental
weights of the simply connected group $G$ determined by the Lie algebra $\frak{g}$.
Let $\sigma$ be the involution on $\{1,\dots,n\}$ given by
$$\alpha^*\overline{u_i}=\overline{u_{\sigma(i)}}.$$
Consider now the short exact sequence
\beg{emi1}{1\r G\r G\rtimes \Z/2\r \Z/2\r 1.
}
By $\Z/2_+$, we shall mean the suspension spectrum of the $G\rtimes\Z/2$-space $\Z/2_+$ by the
action \rref{emi1}. We define $S^A$ by the cofibration sequence
$$\diagram\Z/2_+\rto^\iota & S^0\r S^A
\enddiagram
$$
where $S^0$ is the $G\rtimes \Z/2$-sphere spectrum and $\iota$ is the collapse map
(for terminology, see \cite{lms}). We have, of course,
$$K^{*}_{G\rtimes\Z/2}S^0=R(G\rtimes\Z/2)_{even},$$
$$K^{*}_{G\rtimes\Z/2}\Z/2_+=R(G)_{even}.$$
Here, the subscript $?_{even}$ means that the given $R(G\rtimes\Z/2)$-module is located in
the even dimension of the $\Z/2$-graded ring $K^*$. Furthermore, we have an exact sequence
$$
0\r  K^{0}_{G\rtimes\Z/2}S^A\r R(G\rtimes \Z/2)
\r R(G)\r K^{1}_{G\rtimes \Z/2}S^A\r 0
$$
where the middle arrow is restriction. Therefore, 
$K^{1}_{G\rtimes\Z/2}S^A$ is the free abelian group on irreducible $G$-representations which
do not extend to $G\rtimes \Z/2$. Recall that $\Z/2$ acts on the set of isomorphism
classes of irreducible representations of
$G$; $R(G\rtimes\Z/2)$ is the free abelian group on the regular orbits, and on two copies
of each fixed orbit. Therefore, $K^{0}_{G\rtimes\Z/2}S^A$ can be thought of as the free
abelian group on irreducible $G$-representations which do extend to $G\rtimes \Z/2$-representations.
Equivalently,
$$K^{0}_{G\rtimes\Z/2}S^A=\Z\{u\in T^*\; \text{dominant}\;|\:\alpha^*\overline{u}=\overline{u}\},$$
$$K^{1}_{G\rtimes\Z/2}S^A=\Z\{\text{regular $\alpha^*$-orbits
of dominant weights}\}.$$
Let $S^{(\epsilon)}$ for $\epsilon\in \Z$ denote $S^{A-1}=\Sigma^{-1}S^A$ 
resp. $S^0$ depending on whether $\epsilon$
is odd or even.

\vspace{3mm}
Let $\succ$ denote any chosen linear ordering of the set of subsets of
$\{1,\dots,n\}$.
Let $I_\sigma$ be the set of subsets 
\beg{edefisigma}{\{i_1<\dots<i_k\}\subseteq \{1,\dots,n\}}
such that
$$\{\sigma(i_1),\dots,\sigma(i_k)\}\succ\{i_1,\dots,i_k\}$$
and let $J_\sigma$ be the set of subsets \rref{edefisigma} such that
$$\{\sigma(i_1),\dots,\sigma(i_k)\}=\{i_1,\dots,i_k\}.$$
Let $orb( S)$ for a $\sigma$-invariant set $S$ denote the 
number of regular (=$\Z/2$-free) $\sigma$-orbits of $S$ when $a$ acts on $G$ by $\alpha$, and
of all $\sigma$-orbits of $S$ when $a$ acts on $G$ by $\gamma$.

\vspace{3mm}

\begin{theorem}\label{t1}
There exists an isomorphism of $R(G\rtimes\Z/2)$-modules
\beg{emi2}{\begin{array}{l}K^{*}_{G\rtimes\Z/2}(G)\cong\\[3ex]
K^{*}_{G\rtimes\Z/2}(
\displaystyle\bigvee_{\{ i_1<\dots<i_k\}\in I_\sigma}
%\begin{array}[t]{c}1\leq i_1<\dots<i_k\leq n\\
%\sigma\{i_1,\dots,i_k\}\succ\{i_1,\dots,i_k\}
%\end{array}} 
\Sigma^k\Z/2_+
%\\[4ex]
\vee
\displaystyle\bigvee_{\{ i_1<\dots<i_k\}\in J_\sigma}
%\begin{array}[t]{c}1\leq i_1<\dots<i_k\leq n\\
%\sigma\{i_1,\dots,i_k\}=\{i_1,\dots,i_k\}
%\end{array}} 
\Sigma^k S^{(orb\{i_1,\dots,i_k\})})
\end{array}}
$G\rtimes \Z/2$
acts on the wedge summands on the right hand side of \rref{emi2} through the projection to $\Z/2$.
\end{theorem}

\begin{proof}
We first construct a $G\rtimes\Z/2$-equivariant stable map $u_S$ of each wedge summand
of \rref{emi2} into the $E_\infty$-algebra
\beg{emi3}{F(\Lambda_+,K_{G\rtimes \Z/2})
}
where $\Lambda$ is $G$ on which $g\in G\subset G\rtimes \Z/2$ acts by conjugation 
and $\alpha$ acts by $\gamma$ such that the wedge of all the maps \rref{emi3} induces
an isomorphism on $G$-equivariant coefficients. Here $F(?,?)$, as usual,
denotes the (equivariant) function spectrum, see \cite{lms}.

First, recall the isomorphism \cite{brylinski} 
\beg{ecomment1}{
\begin{array}{l}
\pi_*(F(\Lambda_+,K_{G\rtimes\Z/2})^G)=K_{G}^{*}(G)\\[2ex]
\cong \Omega_{R(G)/\Z}=\Z[\overline{u_1},\dots,\overline{u_n}]
\otimes \Lambda[d\overline{u_1},\dots,d\overline{u_n}]
\end{array}
}
of Theorem \ref{tbz}, induced by \rref{ebz1}.

For regular (=$\sigma$-free) orbits, the map we need
follows from $G$-equivariant considerations: Send, $G$-equivariantly,
$$S^k\r F(\Lambda_+,K_{G\rtimes\Z/2})$$
by the generator 
$$d\overline{u_{i_1}}\wedge\dots\wedge d\overline{u_{i_k}}$$
of \rref{ecomment1} and then
use the fact that $(G\rtimes\Z/2)\rtimes?$ is the left adjoint to the forgetful functor from 
$G\rtimes\Z/2$-spectra to $G$-spectra (cf. \cite{lms}).

Next, for $\sigma$-invariant sets $1\leq i_1<\dots<i_k\leq n$ which consist of a single
orbit, we have $k\leq 2$. If $k=1$, the map follows from Proposition \ref{pl2}. If $k=2$, we have 
a $G$-equivariant map 
$$u:S^1\r F(\Lambda_+,K_{G\rtimes \Z/2})$$
given as the generator $d\overline{u_{i_1}}\wedge d\overline{u_{i_2}}$ of
$\pi_*(F(\Lambda_+,K_{G\rtimes\Z/2})^G)=K_{G}^{*}(G)$.
The $G\rtimes\Z/2$-equivariant map
$$u_{\{i_1,i_2\}}:S^{1+A}\r F(\Lambda_+,K_{G\rtimes\Z/2})$$
we seek may then be defined as
$$N_{G}^{G\rtimes\Z/2}u$$
where $N$ is the multiplicative norm (see \cite{gmc,hhr}). Finally, we may define
$$u_{S_1\amalg\dots\amalg S_\ell}:=u_{S_1}\wedge\dots\wedge u_{S_\ell},$$
using Bott periodicity to identify $S^2$ with $S^{2A}$. Thus, taking a wedge
sum of these maps, we have a map
\beg{emi4}{\begin{array}{c}
X:=F(\bigvee \Sigma^k\Z/2_+\vee\bigvee \Sigma^kS^{(orb\{i_1,\dots,i_k\})},K_{G\rtimes\Z/2})\\[3ex]
\downarrow f\\[3ex]
Y:=F(\Lambda_+,K_{G\rtimes\Z/2})
\end{array}
}
of $K_{G\rtimes\Z/2}$-modules, inducing an isomorphism of $G$-equivariant coefficients (using the
Wirthm\"{u}ller isomorphism \cite{lms} and, again, Bott periodicity). This implies that \rref{emi4} induces
an equivalence on Borel cohomology:
\beg{emi5}{\diagram
F(E\Z/2_+,f^G):F(E\Z/2_+,X^G)^{\Z/2}\rto^(.6)\sim &
F(E\Z/2_+,Y^G)^{\Z/2}.
\enddiagram
}
We need to conclude that \rref{emi4} induces an equivalence on 
$G\rtimes\Z/2$-fixed points, i.e. that we have an equivalence
\beg{emi6}{\diagram(f^G)^{\Z/2}:(X^G)^{\Z/2}
\rto^(.6)\sim &
(Y^G)^{\Z/2}.
\enddiagram
}
To this end, consider $X^G$, $Y^G$ as $\Z/2$-equivariant spectra.

\vspace{3mm}
\begin{lemma}
\label{lmi1}
Denote
$$R=R_G:=(R(G\rtimes\Z/2)/ind_{G}^{G\rtimes\Z/2}R(G))\otimes\Q,$$
$$\widehat{R}=\widehat{R_G}:=(R(G\rtimes\Z/2)/ind_{G}^{G\rtimes\Z/2}R(G))^{\wedge}_{2}\otimes\Q.$$
Then for the $\Z/2$-spectra $Z=X^G,Y^G$, the spectra $\Phi^{\Z/2}Z$, $\widehat{Z}$
(see \cite{lms, gmt})
are rational, and we have an isomorphism
\beg{emi7}{\widehat{Z}_*\cong(\Phi^{\Z/2}Z)_*\otimes_R \widehat{R}}
natural with respect to the map \rref{emi4}. (Here $ind_{G}^{G\rtimes \Z/2}:R(G)\r R(G\rtimes \Z/2)$
denotes the induction, and $(?)^{\wedge}_{2}$ denotes completion at $2$.)
\end{lemma}

\vspace{3mm}

\noindent
{\em Proof of \rref{emi6} using Lemma \ref{lmi1}:} Note that \rref{emi4} also implies an equivalence on
Borel homology:
\beg{emi8}{\diagram
(E\Z/2_+\wedge f^G):(E\Z/2_+\wedge X^G)^{\Z/2}\rto^(.6)\sim &
(E\Z/2_+\wedge Y^G)^{\Z/2}
\enddiagram
}
and Tate cohomology
\beg{emi9}{\diagram\widehat{f^G}:\widehat{X^G}\rto^\sim &\widehat{Y^G}.
\enddiagram
}
By \rref{emi7}, however, the map
$$\Phi^{\Z/2}f^G:\Phi^{\Z/2}(X^G)\r\Phi^{\Z/2}(Y^G)$$
is also an equivalence, and together with \rref{emi8}, this implies \rref{emi6}.
\qed

\vspace{3mm}

\noindent
{\em Proof of Lemma \ref{lmi1}:} The spectra $\Phi^{\Z/2}(M)$, $\widehat{M}$ are rational
for any  cell module $M$ over the $E_\infty$-ring spectrum $K_{\Z/2}$ by 
a theorem of Greenless and May \cite{gmt} which asserts this for $M=K_{\Z/2}$.
Additionally, the methods of \cite{gmt} (or a direct calculation) readily imply that
\beg{emi10}{\begin{array}{l}
\Phi^{\Z/2}(K_{G\rtimes \Z/2}^{G})_*=R,\\
(\widehat{K_{G\rtimes \Z/2}^{G}})_*=\widehat{R}.
\end{array}
}
Now $?\otimes_R \widehat{R}$ is clearly an exact functor on $R$-modules, so by using
the long exact sequence in cohomology, it suffices to filter $X^G$, $Y^G$ both into finite
sequences of cofibrations such that the quotients $Z$ satisfy \rref{emi7}.

In the case of $X^G$, the quotients are either of the form $F(\Z/2_+,K_{G\rtimes\Z/2}^{G})$
for which the statement is trivial (both the geometric and Tate theory are $0$) or
$K_{G\rtimes\Z/2}^{G}$, which is covered by \rref{emi10}, or $\Sigma^A(K_{G\rtimes\Z/2}^{G})$,
which is a cofiber of modules of the first two types.

In the case of $Y^G$, use the decomposition of $\Lambda$ into $G$-orbits with respect
to conjugation of skeleta of the fundamental
alcove, applying also the fact that $\gamma$ acts trivially on $T$. This is, in fact, a
$G\rtimes\Z/2$-CW decomposition, where the cells are of type $H\rtimes\Z/2$ where $H$
is a compact Lie subgroup of $G$ associated to a sub-diagram of the affine Dynkin diagram.

Applying the computation of geometric and Tate $\Z/2$-fixed points of $K_H$, we are done if we can prove
\beg{emi11}{\widehat{R_H}=R_H\otimes_{R_G}\widehat{R_G}.
}
To this end, put
$$R_{H}^{0}:=R(H\rtimes\Z/2)/ind_{H}^{H\rtimes\Z/2}R(H).$$
Recall from \cite{segalrep} that $R(G\rtimes\Z/2)$ is a Noetherian ring 
and $R(H\rtimes\Z/2)$ is a finite module. Therefore, $R^{0}_{H}$ is a finite $R^{0}_{G}$-module.

Now for any Noetherian ring $P$ and a finite $P$-module $M$, we have
\beg{emi12}{M^{\wedge}_{2}=M\otimes_P P^{\wedge}_{2}.}
(Consider the presentation 
$$\bigoplus_n P \r \bigoplus_m P \r M \r 0$$
and right exactness of $(?)^{\wedge}_{2}$ in this case.) 
Rationalizing, \rref{emi12} implies \rref{emi11}.
\end{proof}

\vspace{3mm}

\section{Concrete computations and examples}\label{sconc}

Let, again, $G$ be a simply connected simple compact Lie
group and the generator $a$ of $\Z/2$ 
act on the target $G$ by $\alpha$ or by $\gamma$.
To calculate $K_{G\rtimes \Z/2}(G)$ as an $R(G\rtimes \Z/2)$-module,
 in view of Theorem \ref{t1}, it suffices to calculate the action of the
automorphism $\alpha$ on the Weyl group orbits of the fundamental weights
of the group $G$.

The key observation is that if $\alpha$ is an inner automorphism, then the
action is trivial simply because an inner automorphism does not change
the isomorphism class of a representation. 

Outer automorphisms of simply connected simple Lie groups correspond
to automorphisms of the Dynkin diagram, and therefore are necessarily
trivial for all types except $A,D$ and $E_6$. Furthermore, the
permutation representation of $\Z/2$ on orbits of fundamental weights is isomorphic
to the permutation representation on the set of simple roots (using the bijection
between fundamental weights and simple roots: A fundamental weight is
a point of the weight lattice with which one simple root has minimal positive
inner product, and the other simple roots have inner product $0$).

Recall also that an automorphism $\alpha$ of a semisimple Lie algebra $\frak{g}$
is outer if and only if $rank(\frak{g}^\alpha)<rank(\frak{g})$
(cf., \cite{helgason}). 
From the point of view of symmetric pairs $(\frak{g}, \frak{g}^\alpha)$ of compact type,
we refer to the classification of such pairs (cf., \cite{helgason}, pp. 532-534).

For types $AI$ and $AII$ (corresponding to compact simply connected symmetric
spaces $SU(n)/SO(n)$, $Sp(2n)/SO(n)$), the automorphism is outer,
so the simple weights $v_1,\dots,v_{n-1}$ are transformed by
$$\alpha(\overline{v_i})=\overline{v_{n-i}}.$$
For types $AIII$ and $AIV$ (corresponding to simply connected compact 
symmetric spaces $U(p+q)/U(p)\times U(q)$, the automorphism $\alpha$ is
inner, so all the fundamental weights are fixed.

For types $DI$-$DIII$ (corresponding to compact simply connected symmetric
spaces of type $SO(p+q)/SO(p)\times SO(q)$ with $p+q$ even),
the automorphism $\alpha$ is outer if $p$ (or, equivalently, $q$) is odd,
in which case $\alpha$ interchanges the two fundamental weights
corresponding to the spin representations. When $p$ (or, equivalently, $q$)
is even, the automorphism $\alpha$ is inner and thus, again, the action
is trivial.

For $E_6$, the four different compact simply connected symmetric spaces
are $EI$, $EII$, $EIII$, $EIV$. The automorphism $\alpha$ is outer for
$EI$ and $EIV$, interchanging two pairs of fundamental weights
(and leaving two fundamental weights fixed). For $EII$, $EIII$, the
automorphism is inner and all the fundamental weights are fixed.

\vspace{3mm}
\subsection{Two explicit examples} \label{ssex}
Let us work out explicitly the cases of $G=SU(2)$, $G=SU(3)$
where the involution $\alpha$ is of type $AI$, and $a$ acts on $G$ by 
$\gamma$. In the case of $G=SU(2)$,
denote by $x$ the fundamental weight and by $z$ the representation with character $x+x^{-1}$,
as well as its chosen extension to $SU(2)\rtimes \Z/2$. Denote further by $q$ the complex
sign representation of $\Z/2$. Then 
$$K_{SU(2)\rtimes \Z/2}^{*}(S^0)=R(SU(2)\rtimes\Z/2)=\Z[z,q]/(q^2-1)_{even}$$
and
$$K_{SU(2)\rtimes\Z/2}^{*}(S^A)=Ker(res:R(SU(2)\rtimes\Z/2)\r R(SU(2)))=\Z[z]_{even},$$
generated by $q-1\in R(SU(2)\rtimes\Z/2)$. The argument of K-theory on the right hand
side of \rref{emi2} is 
$$S^0\vee S^A.$$
Thus, we have
$$K^{*}_{SU(2)\rtimes\Z/2}(SU(2))=(\Z[z,q]/(q^2-1)\oplus \Z[z])_{even}
$$
as a $\Z[z,q]/(q^2-1)$-module.

For $G=SU(3)$, we have
$$K_{SU(3)}^{*}(*)=R(SU(3))=\Z[z,t]$$
where $z,t$ are sums of orbits of the two fundamental weights, which
are formed by vertices of the two smallest equilateral triangles with center $0$
in the honeycomb lattice. We have
$$K_{SU(3)\rtimes\Z/2}^{*}(*)=R(SU(3)\rtimes\Z/2)=\Z[\sigma_1,\sigma_2,q]/(q^2-1,(q-1)\sigma_1)$$
where $\sigma_i$ are the elementary symmetric polynomials in $z,t$ and $q$ is the complex 
sign representation of $\Z/2$ (more precisely, a non-canonical choice has to be made in lifting $\sigma_2$ to
a representation of $SU(3)\rtimes \Z/2$, but that is not important for our calculation). To compute
$K^{*}_{SU(3)\rtimes\Z/2}(S^A)$, consider the exact sequence
$$
\diagram
0\dto\\
\Z[\sigma_2]\dto^{\displaystyle q-1}\\
\Z[\sigma_1,\sigma_2,q]/(q^2-1,(q-1)\sigma_1)\dto^{\displaystyle q\mapsto 1}\\
\Z[z,t]\dto^{\scriptstyle\protect\begin{array}{l}\protect z\mapsto 1\\ \protect t\mapsto -1\end{array}}\\
\Z[\sigma_1,\sigma_2]\{z\}\dto\\
0.
\enddiagram
$$
The middle arrow is the restriction 
$$K_{SU(3)\rtimes\Z/2}^{0}(S^0)\r K_{SU(3)}^{0}(S^0),
$$
so the kernel resp. cokernel is $K^{0}_{SU(3)\rtimes\Z/2}(S^A)$
resp. $K^{1}_{SU(3)\rtimes\Z/2}(S^A)$. The argument of K-theory on the 
right hand side of \rref{emi2} is
$$S^0\vee \Sigma \Z/2_+\vee S^{1+A},$$
so we have for the maximal rank pair $(\frak{su(3)}, \frak{h})$:
$$K^{0}_{SU(3)\rtimes\Z/2}(SU(3))=\Z[\sigma_1,\sigma_2,q]/(q^2-1,(q-1)\sigma_1)\oplus \Z[\sigma_1,\sigma_2],$$
$$K^{1}_{SU(3)\rtimes\Z/2}(SU(3))=\Z[z,t]\oplus \Z[\sigma_2]$$
as $R(SU(3)\rtimes\Z/2)=\Z[\sigma_1,\sigma_2,q]/(q^2-1,(q-1)\sigma_1)$-modules.

\vspace{3mm}

\section{Representation-theoretical interpretation of $K_G(G)$}
\label{srep}

Freed, Hopkins and Teleman \cite{fht} showed that for a ``regular'' twisting $\tau$,
$K^{*}_{G,\tau}(G)$ is the free abelian group on irreducible lowest weight
representations of level $\tau-h^\vee$ of the universal central extension $\widetilde{LG}$
of the loop group $LG$. Therefore, $0$ twisting (which is not regular) corresponds to the
critical level.

We found that $K^{*}_{G}(G)$ is, indeed, related to representations of $LG$, but
not lowest weight representations (in the sense that they would be quotients of
the corresponding vertex algebra - see, e.g. \cite{langlands}). 

Instead, we encounter finite representations. Denote by $e_x:LG\r G$ the evaluation
at a point $x\in S^1$. Call a finite-dimensional complex representation of $LG$ {\em finite}
if it factors through a projection of the form
\beg{esrep1}{e_{x_1}\times\dots\times e_{x_n}: LG\r G^n.}
These representations are briefly mentioned in the book \cite{ps}. It is possible to conjecture
that all finite-dimensional representations of $LG$ are finite, although this may depend on
what kind of loops we consider; in this paper, we restrict attention to {\em continuous} loops
with the compact-open topology.

Let us first define the {\em finite representation space} $Rep(\Gamma)$ of a topological group $\Gamma$. 
Since we are about to do homotopy theory, let us work in the category of compactly generated spaces.
A finite-dimensional representation of $\Gamma$ is a finite-dimensional complex vector space $V$
together with a continuous homomorphism
\beg{esrep1a}{\Gamma\r GL(V).}
Continuous homomorphisms are, in particular, continuous maps.
Thus, the set of all representations \rref{esrep1a} for a fixed $V$ forms a topological space, denoted by
$Rep(\Gamma, V)$ with respect to the compact-open topology made compactly generated
(cf.\cite{may}). Consider the topological category $C(\Gamma)$ (both objects and morphisms
are compactly generated spaces, the source $S$
and the target $T$ are fibrations and $Id$ is a cofibration) with objects
\beg{esrep2}{\coprod_V Rep(\Gamma,V)
}
and morphisms
\beg{esrep3}{\coprod_{V,W} Rep(\Gamma,V)\times GL(V,W)
}
(where, say, $S$ is the projection and $T(\phi)$ is the representation on $W$ given
by conjugating the representation on $V$ by $\phi$; $GL(V,W)=\emptyset$
when $dim(V)\neq dim(W)$). Define the {\em representation space}
$Rep(\Gamma)$ as the bar construction on the category $C(\Gamma)$.

For $\Gamma=LG$, let $Rep_0(LG,V)$ denote the subspace of $Rep(LG,V)$ (with the induced
topology made compactly generated) consisting of finite representations. Let $C_0(LG)$
be the subcategory of $C(LG)$ defined by replacing $Rep(LG,V)$ with $Rep_0(LG,V)$,
and let $Rep_0(LG)$ be the bar construction on $C_0(LG)$. Our first aim is to identify a group
completion of the
weak homotopy type of $C_0(LG)$ with $K_G(G)$. To this end, we need a few technical
tools. First of all, let $IG$ denote the group of continuous paths
$\omega:[0,1]=I\r G$ with the 
compact-open topology, and define $C_0(IG)$, $Rep_0(IG)$ analogously with the
above, replacing $LG$ with $IG$. 

\begin{lemma}\label{lrep1}
The inclusion $G\r IG$ via constant maps induces a homotopy equivalence
$$\iota:Rep(G)\r Rep_0(IG).$$
\end{lemma}

\begin{proof}
Define a map
$$\kappa:Rep_0(IG)\r Rep(G)$$
by composing with $e_0$. Then we have $\kappa\iota=Id$. 
Consider now the homotopy $h_t:IG\r IG$ given on $f:I\r G$ by $h_t(f)(x)=f(tx)$.
One easily checks that $h_t$ induces a homotopy on $Rep_0(IG)$ between $\iota\kappa$
and $Id$.
\end{proof}

Next, we define a simplicial symmetric monoidal category as follows: On the simplicial
$n$-level, we take the category $Rep_0(IG\times G^n)$ by which we mean the subcategory
of $Rep(IG\times G^n)$ on representations which factor through evaluation of $IG$
on finitely many points. We let degeneracies be given by projection 
$$IG\times G^{n+1}\r IG\times G^n,$$
and faces by the maps
$$\diagram IG\rto^{ev_1} &G
\rto^\Delta & G\times G,
\enddiagram$$
$$\diagram G
\rto^\Delta & G\times G,
\enddiagram$$
$$\diagram IG\rto^{ev_0} &G
\rto^\Delta & G\times G,
\enddiagram$$
where $\Delta$ is the diagonal. It makes sense to denote this simplicial symmetric monoidal
category by $CH_{\C_2}(C_0(IG),C(G))$ in reference to a ``Hochschild homology complex'',
and its realization by 
\beg{esrephh1}{CH_{Rep(\{e\})}(Rep_0(IG), Rep(G)).}
In fact, we will also be interested in the corresponding spaces 
$$CH_{\C_2}(C(G),C(G)),$$ 
and their realizations
\beg{esrephh2}{CH_{Rep(\{e\})}(Rep(G),Rep(G))}
defined analogously replacing $IG$  by the subgroup of constant paths in $G$. 
Note also that both spaces \rref{esrephh1}, \rref{esrephh2} are $E_\infty$ spaces,
since they are classifying spaces of symmetric monoidal categories.
Lemma
\ref{lrep1} then immediately extends to the following

\begin{lemma}\label{lrep2}
Inclusion of constant loops induces a homotopy equivalence of $E_\infty$ spaces
\beg{esrephh3}{CH_{Rep(\{e\})}(Rep_0(IG), Rep(G))\r CH_{Rep(\{e\})}(Rep(G), Rep(G)).
}
\end{lemma}
\qed

\vspace{3mm}

Next, we prove 

\begin{lemma}\label{lrep3}
There is an equivalence of $E_\infty$ spaces
$$\Omega B CH_{Rep(\{e\})}(Rep(G), Rep(G))\sim K_G(G)_0$$
where the subscript denotes infinite loop space, and by $K_G(G)$ we denote the 
spectrum of $G$-equivariant maps $F_G(G_+,K_G)$ (where, as before, on the
source $G$ acts by conjugation). (See \cite{lms} for the standard
notation.)
\end{lemma}

\begin{proof}
We need a symmetric monoidal functor from the simplicial realization of
the category $CH_{C(\{e\})}(C(G), C(G))$ to vector $G$-bundles on $G$.
Since the source however has topologized objects, it is convenient to
consider an equivalent model of the category of bundles where objects
can vary continuously parametrized by a space. In more detail, we consider
the category with both objects and morphisms topologized, where 
the objects are formed by the disjoint union of CW-complexes $X\times\{\xi\}$ where $\xi$ is a
$G$-bundle on $G\times X$. Morphisms are disjoint unions of spaces $\Gamma_{X,\xi,Y,\eta}$
consisting of triples $x\in X$, $y\in Y$ and isomorphisms $f:\xi|G\times\{x\}\r \eta|G\times \{y\}$
topologized so that the projection $\Gamma_{X,\xi,Y\eta}\r X\times Y$ is locally a product
(which is done canonically by local triviality of $\xi$, $\eta$).
This is a symmetric monoidal category (which we will denote by $Bun_{G}(G)^\prime$
by Whitney sum of pullbacks of $G$-bundles over $G\times X$ and
$G\times Y$ to a $G$-bundle over $G\times X\times Y$, and moreover a groupoid whose
skeleton is the ordinary symmetric monoidal category $Bun_G(G)$ of $G$-bundles on $G$.

Then,
if we denote by $C$ the simplicial realization of the category $CH_{C(\{e\})}(C(G), C(G))$,
it suffices to construct a symmetric monoidal functor 
$C\r Bun_G(G)^\prime$, which can be constructed from a $G$-bundle on $G\times Obj(C)$ which satisfies 
the appropriate additivity and functoriality properties. To this end,
it suffices to construct a functor, symmetric monoidal over $G$, of the form
\beg{erephh4}{G\times CH_{C(\{e\})}(C(G), C(G))\r C(G),
}
where in the source, we consider the ``total'' category spanned by the level
morphisms as well as the simplicial structure. Construct \rref{erephh4}
as follows: on each level, put
$$(g,V)\mapsto V.$$
All faces and degeneracies are set to the identity, except the $0$'th face on each
level. The $0$-face from level $n$ to level $n-1$
 is sent, at $g\in G$, to multiplication  by 
$$(1,g,\underbrace{1,\dots,1)}_{\text{$n-1$ times}}.$$ 
Applying the classifying
space functor, this gives an $E_\infty$-map from $CH_{Rep(G)}(Rep(G), Rep(G))$
to the space of $G$-equivariant vector bundles on $G$. Applying an infinite loop
space machine and localization at the Bott element, we obtain a map which, up to homotopy, can be expressed as
\beg{erephh5}{CH_K(K_G,K_G)\r K_G(G)}
where $K_G$ denotes the $K$-module of $G$-fixed points of $G$-equivariant $K$-theory.
Now there is a spectral sequence (coming from the simplicial structure) converging to the
homotopy of the left hand side of \rref{erephh5} whose $E_2$-term is 
\beg{erephh6}{HH_\Z(R(G),R(G)).}
Note that the source and target coefficient rings of \rref{erephh5}, as well as \rref{erephh6} are rings,
where in fact the ring \rref{erephh6} is isomorphic to
the homotopy ring of the target of \rref{erephh5}, and the generators of \rref{erephh6} are permanent cycles, which map
to the corresponding generators of $K_{G}^{*}(G)$ by Theorem \rref{tbz}.
Thus, we are done if we can prove that \rref{erephh5} is a map of ring spectra. 

In fact, one can rigidify \rref{erephh6} to become
a map of $E_\infty$ ring spectra. The functor from $CH_{Rep(G)}(Rep(G), Rep(G))$ to
$Bun_G(G)^\prime $ is a weak symmetric bimonoidal functor on each simplicial level,
with the simplicial structure maps weakly preserving the structure. Thus, on the ``totalized'' category
where we combine the levels and consider simplicial structure maps as morphisms, we obtain
a weak symmetric bimonoidal functor into $Bun_G(G)^\prime$. Applying the Elmendorf-Mandell machine
\cite{em} and localization at the Bott element, an $E_\infty$ model of \rref{erephh5} follows.
(The Joyal-Street construction on categories with both objects and morphisms topological \cite{kl}
is also relevant.)
\end{proof}

\vspace{3mm}

Finally, using the map
\beg{elli}{LG\r IG\times G^n}
given by $f\mapsto (f\circ\pi,f(0),\dots,f(0))$ where $\pi:I\r I/0\sim 1=S^1$ is the projection,
we obtain a map
\beg{erephh7}{p:CH_{Rep(G)}(Rep_0(IG),Rep(G))\r Rep_0(LG).}

\begin{lemma}\label{lrep4}
The map $p$ is an equivalence.
\end{lemma}

\begin{proof}
First, we show that $p$ is a quasi-fibration. We use the criterion in \cite{dt}, which
is restated as Theorem 2.6 of \cite{mq}. We first observe that $Rep_0(G)$
is a disjoint sum of connected components indexed by dimension $d$ of the
representation. We may consider one connected component at
a time. For a given $d$, let the $i$'th (increasing) filtered part be
spanned by all representations which are (up to isomorphism) of the form $V\otimes W$ where 
$W$ has dimension $\geq n-i$ and factors through the projection
$$e_0: LG\r G.$$
The open neighborhoods required in Theorem 2.6 of \cite{mq} are
then spanned by representations of the form $V\otimes W$ where
$W$ is of dimension $\geq n-i$, and factors through the projection
of $LG$ to $Map(U,G)$ where $U$ is an $\epsilon$-neighborhood
of $1$ in $S^1$. The homotopies $H_t$ and $h_t$ are then
defined by contracting $U$ to $1$.

Once we know that $p$ is a quasifibration, the statement follows, as it
is easily checked that the inverse image of every point is contractible.
\end{proof}

\vspace{3mm}
Putting together Lemmas \ref{lrep1}, \ref{lrep2}, \ref{lrep3}, \ref{lrep4}, we
now obtain the following

\begin{theorem}
\label{trrep}
The group completion of the $E_\infty$ space $Rep_0(LG)$ is weakly equivalent 
to the infinite loop space $K_G(G)_0$.
\end{theorem}
\qed

\vspace{3mm}

\section{Representation-theoretical interpretation of $K_{G\rtimes\Z/2}(G)$}

\label{srep2}
There is also an equivariant version of these constructions with respect
to a $\Z/2$-action where the generator of $\Z/2$ acts on $G$ either by 
$\alpha$ or by $\gamma$. In these cases, we consider the topological
group $LG\rtimes\Z/2$.
$\Z/2$ acts on the loop $f:S^1\r G$ by $f\mapsto g$ where
 in the case of action by $\alpha$ on the target, $g(t)=\alpha(f(t))$,
and in the case of action by $\gamma$ on the target, 
$g(t)=\alpha(f(1-t))$ (again, we use the identification $S^1=I/0\sim 1$).
In both cases, {\em finite representations} are defined as finite-dimensional
representations which factor through a projection to $G^n\rtimes \Z/2$
by evaluation at finitely many points (where, in the case of the $\gamma$-action
on the target, with each evaluation point $t$, we must also include $1-t$).
Now restricting, again, to finite representations, we obtain from $C(G\rtimes\Z/2)$
the categories
$C_0(LG\rtimes \Z/2)$ and their classifying spaces $Rep_0(LH\rtimes\Z/2)$.

Next, we may also similarly define 
\beg{eeqcc}{C_0((IG\times G^n)\rtimes \Z/2),} 
and
its classifying space
\beg{eeqhh}{Rep_0(IG\times G^n)\rtimes\Z/2).}
Here the generator $a$ of $\Z/2$ acts on $f:I\r G$
by $f\mapsto g$ where $g(t)=\alpha(f(t))$,
in the case of $\alpha$-action on the target, $a$ acts on each of the $n$ copies
of $G$ separately by $\alpha$, in the case of $\gamma$ action on the target, 
$a$ acts on $G^n$ by
$(g_1,\dots,g_n)\mapsto (\alpha(g_n),\dots,\alpha(g_1))$.
In the case where $a$ acts by $\alpha$ on the target $G$, \rref{eeqhh}
forms a simplicial category, and its classifying space \rref{eeqhh} forms
a simplicial space. In the case of $a$ acting on the target $G$ by $\gamma$,
the $\Z/2$-action is an automorphism over the involution of the simplicial
category reversing the order of each set $\{0,\dots,n\}$, so we can
still form a ``simplicial realization'' by letting $a$ act on each standard
simplex $\Delta_n$ by $[t_0,\dots,t_n]\mapsto[t_n,\dots,t_0]$ in
barycentric coordinates. In both cases, we denote the simplicial realizations
$$CH_{\C_2}(C_0I(G),C(G))_{\Z/2}$$
and
$$CH_{Rep(G)}(Rep_0(IG),Rep(G))_{\Z/2}.$$
Again, the precise application of the method of proof of Lemma \ref{lrep1} gives

\begin{lemma}
\label{leeh1}
Restriction to constant loops induces a homotopy equivalence
$$CH_{Rep(G)}(Rep_0(IG),Rep(G))_{\Z/2}\r CH_{Rep(G)}(Rep(G),Rep(G))_{\Z/2}$$
(where the target is defined the same way as the source, restricting to constant
maps $I\r G$).
\end{lemma}
\qed

\vspace{3mm}
Next, the map \rref{elli} defines a map
\beg{elli1}{p_{\Z/2}:CH_{Rep(G)}(Rep_0(IG), Rep(G))_{\Z/2}\r Rep_0(LG\rtimes\Z/2)
}
in both the cases of $\alpha$ and $\gamma$ action on the target. Again, we have an
equivariant analogue of Lemma \ref{lrep4}:

\begin{lemma}
\label{leeh2}
The map $p_{\Z/2}$ is an equivalence.
\end{lemma}

\begin{proof}
In both cases, the filtration and the deformations used in the proof of Lemma \ref{lrep4}
have obvious $\Z/2$-equivariant analogous. Therefore, the same argument applies.
\end{proof}

\vspace{3mm}

We shall now construct an  $E_\infty$ map 
\beg{eehh10}{CH_{Rep(G)}(Rep(G),Rep(G))_{\Z/2}\r K_{G\rtimes\Z/2}(G)}
where by the right hand side, we mean the space of $G\rtimes\Z/2$-equivariant maps
$G\r (K_{G\rtimes\Z/2})_0$ where the generator $a$ of $\Z/2$ acts on $G$ by
either $\alpha$ or $\gamma$. To this end, again, it suffices to construct 
a symmetric monoidal map of the realization of the category
$$CH_{\C_2}(C(G),C(G))_{\Z/2}$$
to the category of $G\rtimes\Z/2$-vector bundles on $G$.
In the case of $a$ acting on the target $G$ by $\alpha$, the construction is directly
analogous to Lemma \ref{lrep3}. 

In the case of $a$ acting by $\gamma$, we must deal with the involutive automorphism
of the simplicial category. It is most convenient to give this in the form of a
$G\rtimes\Z/2$-equivariant bundle on 
$$G\times CH_{Rep(G)}(Rep(G),Rep(G))_{\Z/2}.$$
We let the bundle be defined by the same formula as in the non-equivariant
case, with $\Z/2$-action
$$\alpha(g,v,[t_0,\dots,t_n])=(\alpha(g)^{-1}, \alpha(g^nv),[t_n,\dots,t_0])$$
(see the proof of Lemma \ref{lrep3}).

\begin{theorem}
\label{teehh}
The induced map
\beg{ehhp*}{\Omega B(CH_{Rep(G)}(Rep(G),Rep(G))_{\Z/2})\r K_{G\rtimes \Z/2}(G)_0
}
is an equivalence.
\end{theorem}

\begin{proof}
It is useful to note that $Rep(G\rtimes\Z/2)$ is a $\Z/2$-equivariant symmetric monoidal category
in the sense that there is a transfer functor
\beg{ehhp1}{\bigoplus_{\Z/2}:Rep(G\rtimes\Z/2)\r Rep(G\rtimes\Z/2)}
extending the commutativity, associativity and unitality axioms in
the usual sense. The functor \rref{ehhp1} is given by 
$$V\mapsto V\oplus V\otimes A$$
where $A$ is the sign representation of the quotient $\Z/2$. 
We may recover $Rep(G)$ as the image of the functor \rref{ehhp1}, and
then the functor \rref{ehhp1} provides a symmetric monoidal
model of the restriction $Rep(G\rtimes \Z/2)\r Rep(G)$.
Applying this construction level-wise on the simplicial level, 
$$ CH_{Rep(G)}(Rep(G),Rep(G))_{\Z/2}\r K{G\rtimes \Z/2}(G)_0$$
becomes a map of $\Z/2$-equivariant $E_\infty$-spaces, and applying
$\Z/2$-equivariant infinite loop space theory, and localizing at the Bott
element, we obtain a map of $Z\2$-equivariant spectra (indexed over the complete universe)
\beg{ehhp**}{CH_K(K_G,K_G)_{\Z/2}\r K_{G\rtimes\Z/2}(G).}
In fact, it is a map of $K_{\Z/2}$-modules. 
Forgetting the $\Z/2$-equivariant structure, we recover the map \rref{erephh5},
which is, of course, an equivalence. To show that the map \rref{ehhp**}, which is
a non-equivariant equivalence, is, in effect, a $\Z/2$-equivalence, we use a
variant of the finiteness argument of the section \ref{smi}, this time using finiteness
directly over $K_{\Z/2}$. 

Recall again that 
$$R(G)\cong \Z[v_1,\dots,v_r]$$
where $v_i$ are the fundamental irreducible representations of $G$. Let us call 
a finite-dimensional complex representation of $G$ {\em of degree $n$}
if it is a sum of subrepresentations, each of which is isomorphic to a tensor product
of $n$ elements of $\{v_1,\dots,v_r\}$ (an element is allowed to occur more
than once). We see, in fact, that we obtain degree $d$ versions of all the
symmetric monoidal categories considered, giving a stable splitting of the form
$$K_G\simeq \bigvee_{d\geq 0} K_{G}^{d},$$
$$CH_K(K_G,K_G)\simeq \bigvee_{d\geq 0} CH_{K}^{d}(K_G,K_G).$$
Also, these splittings are $\Z/2$-equivariant (since $\Z/2$ acts on $\{v_1,\dots,v_r\}$
by permutation), and in fact the map \rref{ehhp**} decomposes into maps
$$CH_K(K_G,K_G)_{\Z/2}\r K_{G\rtimes\Z/2}(G).$$
On the other hand, the Brylinski-Zhang construction restricts to a map
of finite $K_{\Z/2}$-modules
\beg{ehhpi}{\begin{array}{l}
K^{d}_{G\rtimes\Z/2}\wedge(
\displaystyle\bigvee_{\begin{array}[t]{c}1\leq i_1<\dots<i_k\leq n\\
\sigma\{i_1,\dots,i_k\}\succ\{i_1,\dots,i_k\}
\end{array}} \Sigma^k\Z/2_+
\\[8ex]
\vee
\displaystyle\bigvee_{\begin{array}[t]{c}1\leq i_1<\dots<i_k\leq n\\
\sigma\{i_1,\dots,i_k\}=\{i_1,\dots,i_k\}
\end{array}} \Sigma^k S^{(orb\{i_1,\dots,i_k\})})\r CH_{K}^{d}(K_G,K_G).
\end{array}
}
Since each of $K_{\Z/2}$-modules are finite, an equivariant map which is a non-equivariant
equivalence is an equivalence by the argument of Section \ref{smi} (in fact, made simpler
by the fact that we do not consider $G$-equivariance here). Denoting the left hand side
of \rref{ehhpi} by $\mathcal{K}^d$, we now have a diagram of $K_{\Z/2}$-modules
of the form
\beg{ehhpii}{\diagram
\bigvee_{d\geq 0} \mathcal{K}^d\rto^(.4)\sim\drto_\sim & CH_{K}(K_G,K_G)_{\Z/2}\dto\\
& K_{G\rtimes \Z/2}(G).
\enddiagram
}
The horizontal arrow is a $\Z/2$-equivalence by the fact that \rref{ehhpi} is
a $\Z/2$-equivalence and by the stable splitting, the diagonal arrow is a
$\Z/2$-equivalence by Theorem \ref{t1}. Therefore, the vertical arrow is 
a $\Z/2$-equivalence.
\end{proof}

\vspace{5mm}
\noindent
\address{Po Hu\\Department of Mathematics\\Wayne State University\\1150 Faculty/Administration Bldg.\\
656 W. Kirby\\ Detroit, MI 48202, U.S.A.}

\noindent
\email{po@math.wayne.edu}

\vspace{5mm}
\noindent
\address{Igor Kriz\\Department of Mathematics\\University of Michigan\\530 Church Street\\Ann Arbor, MI 28109-1043,
U.S.A.}

\noindent
\email{ikriz@umich.edu}

\vspace{5mm}
\noindent
\address{Petr Somberg\\
Mathematical Institute, MFF UK\\
Sokolovsk\'{a} 83\\
180 00 Praha 8, Czech Republic}

\noindent
\email{somberg@karlin.mff.cuni.cz}

\noindent{{\bf 2010 MSC classification:} 19L47, 55N15, 55P91, 53C35}

\noindent{{\bf Keywords:} equivariant K-theory, stable homotopy theory,
compact Lie groups, compact symmetric spaces, representations
of loop groups}

\end{document}